\documentclass[11pt,a4paper]{article}
\usepackage{a4,amssymb,amsmath,amsthm,url,color,enumerate,float}
\usepackage{bezier,amsfonts,amssymb,graphicx,amsthm,url}
\usepackage[english]{babel}
\title{Degree Monotone Paths}
\date{}
\begin{document}
\newtheorem{theorem}{Theorem}[section]
\newtheorem{definition}{Definition}[section]
\newtheorem{proposition}[theorem]{Proposition}
\newtheorem{corollary}[theorem]{Corollary}
\newtheorem{lemma}[theorem]{Lemma}

\newtheoremstyle{break}
  {}
  {}
  {\itshape}
  {}
  {\bfseries}
  {.}
  {\newline}
  {}

\theoremstyle{break}

\newtheorem{propskip}[theorem]{Proposition}
\DeclareGraphicsExtensions{.pdf,.png,.jpg}
\author{Yair Caro \\ Department of Mathematics\\ University of Haifa-Oranim \\ Israel \and Josef  Lauri\\ Department of Mathematics \\ University of Malta
\\ Malta \and Christina Zarb \\Department of Mathematics \\University of Malta \\Malta }

\maketitle
\begin{abstract}
We shall study degree-monotone paths in graphs, a problem inspired by the celebrated theorem  of  Erd{\H{o}}s-Szekeres concerning  the longest monotone subsequence of a given sequence of numbers.

A path P in a graph G is said to be a degree monotone path if the sequence of degrees of the vertices in P in the order they appear in P is monotonic.

In this paper we shall consider these three problem related to this parameter:
\begin{enumerate}
\item{ Find bounds on $mp(G)$ in terms of other parameters of $G$.}
\item{Study $f(n,k)$ defined to be the maximum number of edges in a graph on $n$ vertices with $mp(G) < k$.}
\item{Estimate the minimum and the maximum over all graph $G$ on $n$ vertices of $mp(G)+mp(\overline{G})$.}
\end{enumerate}

For the first problem our main tool will be the Gallai-Roy Theorem on directed paths and chromatic number.  We shall also consider in some detail maximal planar and maximal outerplanar graphs in order to investigate the sharpness of the bounds obtained.  For the second problem we establish a close link between $f(n,k)$ and the classical Turan numbers.  For the third problem we establish some Nordhaus-Gaddum type of inequalities.  We conclude by indicating some open problems which our results point to.

\end{abstract}

\section{Introduction}

We shall study degree-monotone paths in graphs, a problem inspired by the celebrated theorem  of  Erd{\H{o}}s-Szekeres concerning  the longest monotone subsequence of a given sequence of numbers \cite{fox2012erdos}.  

A path $P=v_1\ldots v_k$ in a graph $G$ is said to be a \emph{degree monotone path} of length $k$ if $deg(v_1) \leq deg(v_2) \leq \ldots \leq deg(v_k)$ or $deg(v_1) \geq deg(v_2) \geq \ldots \geq deg(v_k)$.

Such degree monotone paths are called uphill and downhill paths in  recent papers based upon  Ph.D. theses that were recently submitted for publication \cite{downhill2,downhill3,updownhilldom}.  The authors' main motivation was to consider problems about downhill and uphill domination, another interesting variant of the many different ideas in dominations that are being studied \cite{dom_survey,haynes1998fundamentals,haynes1998domination}.  

The paper \cite{updownhilldom} explicitly addresses the problem:  what can be said about the maximum length of a downhill, respectively uphill, path in a graph $G$.  In this paper we use the term degree monotone path to unify the idea of an uphill and downhill path.  We define the following parameter:  given a connected graph $G$, $mp(G)$ is the length of the largest degree monotone path in $G$.  The following natural problems concerning this parameter arise:
\begin{enumerate}
\item{Finding bounds on $mp(G)$ in terms of other parameters of the graph G.}

\item{Define $ f(n,k) =\max\{ |E(G)|: |G| = n, mp(G) < k\}$, be the maximum number of edges in a graph on $n$ vertices with no degree monotone path of length $k$.  We are interested in studying $f(n,k)$.}

\item{Estimating $\min\{mp(G) + mp(\overline{G})\}$ and $\max\{mp(G) + mp(\overline{G})\}$ over all graphs on $n$ vertices.  This is a Nordhaus-Gaddum type of problem for the parameter $mp(G)$.  An excellent survey of such results is found in \cite{NGsurvey1}}

\item{How does the value of $mp(G)$ change when $G$ is subjected to various graph operations such as edge/vertex deletion or addition?}
\end{enumerate}

As we shall see, Problem 1  can be treated via the Gallai-Roy theorem \cite{dougwest}, through which we show that $mp(G) \geq \chi(G) \geq \omega(G)$, where $\chi(G)$ is the chromatic number of $G$ and $\omega(G)$ is the clique number, for every graph, and that in general, this is best possible.

We then consider maximal outerplanar graphs, and show, using light-edge techniques \cite{Fabrici,hackmannkemnitz}, that for every maximal outerplanar graph $G$ on at least five vertices, $mp(G) \geq 4$, and this is best possible --- thus showing that the bound $mp(G) \geq \chi(G)$ can be slightly improved for this class of graphs.  On the other hand,  we show by construction that there exist arbitrarily large  maximal planar graphs $G$ with $\chi(G)=mp(G)=4$.

For Problem 2 we recall the definition of the Turan numbers \cite{bollobas2004extremal}: \[t(n,k)  = \max \{|E(G)|: |G| = n, G \mbox{ contains no copy of } K_k\}\]  We shall establish a close connection between $t(n,k)$ and $f(n,k)$, using families of graphs already considered in a remotely related problem by M.Albertson \cite{Albertson92}.

For Problem 3, we get a sharp bound for $\max\{mp(G) + mp(\overline{G})\}$  and close bounds for $\min\{mp(G) + mp(\overline{G})\}$, using the results obtained for Problem 1, and the Nordhous-Gaddum bounds for the chromatic number \cite{nordhaus1956complementary,dougwest}.

Problem 4 involves many possible operations on graphs, hence we defer these types of result to a later paper \cite{CLZ5dmp}.

Lastly, from a complexity point of view, we observe that, since for a regular graph $G$, $mp(G)$ is equivalent to finding the longest path in $G$, it is evident that already in regular graphs, finding $mp(G)$ is an $NP$-hard problem \cite{karger1997approximating}.

\section{Bounds for $mp(G)$}

In this section, we consider lower bounds and upper bounds for $mp(G)$ for a general graph $G$, as well as for the class of maximal outerplanar graphs.  

We first consider the relationship between $mp(G)$ and $\chi(G)$.  For this we use the well-known Gallai-Roy Theorem \cite{dougwest} :
\begin{theorem} [Gallai-Roy] \label{gallai-roy}
In any orientation of a graph $G$, the length of the longest directed path is at least $\chi(G)$.
\end{theorem}

\begin{theorem} \label{mp_chi}
For every graph $G$, $mp(G) \geq \chi(G)$, and the bound is sharp.
\end{theorem}

\begin{proof}
Let us consider a graph $G$ and let us orient it such the edge $uv$ is oriented from $u$ to $v$ is $deg(u) \leq deg(v)$.  By Theorem \ref{gallai-roy},  there is a directed path of length at least $\chi(G)$.  This path is degree monotone by definition of the orientation, and hence $mp(G) \geq \chi(G)$.

The bound is achieved by the following construction:  consider a complete multipartite graph $G$ having $k$ parts all of different sizes, ranging from $\lceil \frac{n}{k} \rceil - \lceil \frac{k}{2} \rceil$ to $\lceil \frac{n}{k} \rceil + \lfloor \frac{k}{2} \rfloor$ .  Then $\chi=k$.  Now let the parts be $X_1,X_2, \ldots,X_k$ such that $|X_1| < |X_2| < \ldots <|X_k|$ - hence the degrees of vertices in $X_1$ are larger than those of vertices in $X_2$, which in turn are larger than those of vertices  in $X_3$  and so on.  Then if we start from a vertex in $X_1$, then take a vertex from $X_2$ and so on, it is clear that we can take exactly one vertex from each part in this order, creating a degree monotone path in non-increasing order.  Hence $mp(G)=k$.

\end{proof}

\begin{corollary}
 For every graph $G$ on $n$ vertices:
\begin{enumerate}
\item{$mp(G) \geq \omega(G)$ and $mp(G) \geq  \frac{n}{\alpha(G)}$}
\item{$\max \{ mp(G),\alpha(G)\} \geq \sqrt{n}$}
\item{If $G$ is $K_{1,r}$-free, for $r \geq 3$, then \[mp(G) \geq \left \lceil \frac{\Delta}{r-1} \right \rceil +1 \geq \frac{\Delta +r -1}{r-1}.\]}
\end{enumerate}
\end{corollary}

\begin{proof}
\mbox{}\\*

\noindent 1.  \indent It is well known that $\chi(G) \geq \omega(G)$ and also $\chi(G) \geq \frac{n}{\alpha(G)}$, and hence by Proposition \ref{mp_chi} the results follow.
\medskip

\noindent 2. \indent We observe that if $\alpha(G) \leq \sqrt{n}$, then $mp(G) \geq \chi(G) \geq \frac{n}{\alpha(G)} \geq \sqrt{n}$, and the result follows.
\medskip

\noindent 3. \indent We observe that if $G$ is $K_{1,r}$-free, then $\chi(G) \geq \left \lceil \frac{\Delta}{r-1} \right \rceil +1 \geq \frac{\Delta +r -1}{r-1}$, since it is clear that a vertex $v$ it can have at most $r-1$ neighbours in any colour class. Therefore letting $v$  have degree $\Delta$, it follows that there must be at least $\left \lceil \frac{\Delta}{r-1} \right \rceil +1 $ colour classes, and hence by Theorem \ref{mp_chi}, \[ mp(G) \geq \chi(G) \geq  \left \lceil \frac{\Delta}{r-1} \right \rceil +1 \geq \frac{\Delta +r -1}{r-1}.\]
\end{proof}

 It is clear that $mp(G)=1$ if and only if $G$ has no edges.  So we now give a characterization of graphs having $mp(G)=2$.  One can see that, for fixed $k$, deciding whether a graph $G$ on $n$ vertices has $mp(G)=k$ can be done by a brute-force algorithm checking all the paths of length $k$ to find one which is a degree monotone path, and checking all paths of length $k+1$ to verify that none of these is a degree monotone path.  The number of paths to be checked  is bounded by $O(n^{k+1})$.

\begin{lemma} \label{mp_two}
Let $G$ be a connected graph on $n \geq 3$ vertices with $mp(G)=2$.  Let $e=(x,y)$ be an edge.  Then
\begin{enumerate}
\item{$deg(x) \neq deg(y)$.}
\item{If $deg(x) > deg(y)$, then for every vertex $u$ in $N(y)$, $deg(u) > deg(y)$.}
\item{If $deg(x) > deg(y)$, then for every vertex $u$ in $N(x)$, $deg(u) < deg(x)$.}
\end{enumerate}
\end{lemma}

\begin{proof}
1. \indent  Suppose $deg(x)=deg(y)$, then since $G$ is connected we may assume without loss of generality that there is a vertex $u$ adjacent to $y$.  If $deg(u) \geq deg(y)$ then the path $x-y-u$ is a degree monotone path of length 3, otherwise $u-x-y$ is such a path, contradicting the fact the $mp(G)=2$.

\noindent 2. \indent Assume $deg(x) > deg(y)$.  Clearly, $x$ is in $N(y)$ and $deg(x) > deg(y)$.  If there is another vertex $w$ in $N(y)$ and $deg(w) \leq deg(y)$, then $x-y-w$ is a degree monotone path of length 3, a contradiction.

\noindent 3. \indent Assume $deg(x) > deg(y)$.  Clearly, $y$ is in $N(x)$ and $deg(y) < deg(x)$.  if there is another vertex $w$ in $N(x)$ and $deg(w) \geq deg(y)$, then $w-x-y$ is a degree monotone path of length 3, again a contradiction.

\end{proof}

\begin{theorem}
Let $G$ be a connected graph on $n \geq 3$ vertices.  Then $mp(G)=2$ if and only if $G$ is a bipartite graph with partition $A \cup B=V(G)$ such that $\forall x \in A$, $deg(x) > \max \{deg(y): y \in N(x)\}$.
\end{theorem}

\begin{proof}
Suppose $mp(G) = 2$.  Then by Proposition \ref{mp_chi}, $mp(G) \geq \chi(G)$, hence we infer that $\chi(G) = 2$, namely $G$ is bipartite.

For every edge $e=(x,y)$ we know by Lemma \ref{mp_two} that, without loss of generality, $deg(x) > deg(y)$ and also that, in this case, $deg(x) > deg(u)$ $\forall u \in N(x)$.
 
So in every edge $e=(x,y)$ with $deg(x) > deg(y)$,  let $x \in A$ and $y \in B$.
 
By Lemma \ref{mp_two}, this is a well-defined partition with $A \cup B = V(G)$, as every vertex is either always of the minimum degree in its closed neighborhood  or the maximum degree in its closed neighborhood.  Clearly, with this partition, there is no edge between any two vertices in the same part, and $ \forall x \in A$, $deg(x) > \max\{deg(y) : y \in  N(x)\}$.

Suppose now that $G$ is a bipartite graph with a partition $A \cup B  = V(G)$  such that $\forall x \in A$, $deg(x) > \max\{ deg(y) : y \in N(x)\}$.
 
If there is a degree monotone path of length 3 then either it starts with $x \in A$ then $y\in B$ and then $z \in A$, or else it starts with $y \in B$ then $x \in A$ and then $z \in B$.  In the former case, $deg(y) < \min\{ deg(x) ,deg(z) \}$ and hence the path is not degree monotone , and in the latter case $deg(x) > \max\{ deg(y) ,deg(z) \}$, and again the path is not degree monotone.  This completes the proof.
\end{proof}

\subsection{Maximal Outerplanar Graphs}

While Proposition \ref{mp_chi} is sharp in general, it is of interest to find cases of graphs in which, non-trivially, $mp(G) > \chi(G)$.

One such family is that of \emph{maximal outerplanar graphs} (MOPs), for which it is well known that $\chi=3$. 

In the sequel, we shall use the following result from ``light-edge theory" \cite{Fabrici,hackmannkemnitz}.

\begin{theorem} [Hackman-Kemnitz] \label{HandK}
Every outerplanar graph $G$ of minimum degree 2 contains an edge $vw$ with $deg(v)=2$ and $deg(w) \leq 3$, or a $3$-path $v,w,x$ with $deg(v)=2$, $deg(w)=4$ and $deg(x)=2$.
\end{theorem}

It is well known that every MOP has at least two vertices of degree two, and that every MOP has a unique hamiltonian cycle, which bounds the exterior region, and that each interior region is a triangle.  An edge which is not on $C$ is called a chord.  If $n \geq 4$, no two vertices of degree two are adjacent.  An edge $e=(x,y)$ is said to be \emph{regular} if $deg(x)=deg(y)$.  We now prove the following lemma:

\begin{lemma} \label {common_neighbour}
Let $G$ be a MOP with $n \geq 3$.  Consider an edge $xy$.  Then if $xy$ is on the unique Hamiltonian cycle which bounds the exterior region, $x$ and $y$ have exactly one common neighbour, while if $xy$ is a chord, $x$ and $y$ have exactly two common neighbours.
\end{lemma}

\begin{proof}
Consider $xy$ an edge on the unique hamiltonian cycle $C$.  Then $xy$ bounds the exterior region as well as an interior triangle.  Hence $x$ and $y$ have a common neighbour $z$, which is on $C$.  If $n=3$, then $x,y,z$ are the only vertices and the graph is $K_3$.  Otherwise either $xz$ or $yz$ is a chord and hence $x$ and $y$ cannot have another common neighbour since they bound the exterior region.

On the other hand, if $xy$ is a chord, that is it is not on $C$ which bounds the exterior region, it must be on the boundary of two interior regions, which in turn must both be triangles.  Hence $x$ and $y$ have two common neighbours.  

Note that any two vertices cannot have more than two common neighbours, since otherwise the vertices and three common neighbours induce $K_{2,3}$ as a subgraph, and the graph is not outerplanar.

\end{proof}

It is easy to see that for $n=4$ and $n=5$, $mp(G)=n-1$.  However the situation is different for $n \geq 6$, as we show in the following Theorem.

\begin{theorem} \label {dm4}

Let $G$ be a MOP on $n \geq 6$ vertices. Then
\begin{enumerate}
\item{ $4 \leq mp(G) \leq n-1$.}
\item{There exist arbitrarily large MOPs with $mp(G)=4$.}
\item{There exist arbitrarily large MOPs with no regular edges and $mp(G)=4$.}
\end{enumerate}
\end{theorem}

\begin{proof}
\mbox{}\\
1. \indent We first consider the lower bound.

Let $G$ be a MOP, with unique Hamiltonian cycle $C$.  Then by Theorem \ref{HandK}, there is either an edge $vw$ with $deg(v)=2$ and $deg(w) \leq 3$, or a $3$-path $v,w,x$ with $deg(v)=2$, $deg(w)=4$ and $deg(x)=2$.  Let us consider these two cases separately:

\noindent \emph{Case 1}:  Consider the edge  $vw$ with $deg(v)=2$ and $deg(w) =3$, therefore $vw$ must be on $C$. Then $v$ and $w$ have a common neighbour $x$, and $wx$ must be a chord.  Hence $w$ and $x$ have another common neighbour by Lemma \ref{common_neighbour}.  Let this common neighbour be $y$.  Now since $n \geq 5$, $deg(x)$ and $deg(y)$ are at least 3.  If $deg(x) \leq deg(y)$, then $vwxy$ is a dmp of length 4.  If $deg(x) > deg(y)$, then $vwyx$ is a dmp of length 4.  In either case there is a dmp of length 4, closing this case.

\begin{figure}[h!]
\centering
\includegraphics{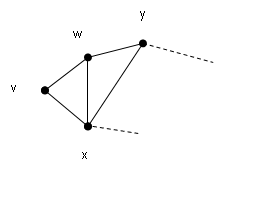}
\caption{Case 1}
\end{figure}

\medskip

\noindent \emph{Case 2}: Consider the $3$-path $v,w,x$ with $deg(v)=2$, $deg(w)=4$ and $deg(x)=2$.  Consider $vw$ which must be an edge on $C$.  Then $v$ and $w$ have a common neighbour $y$, and $wy$ must be a chord.  If $deg(y)=3$, then we have the edge $vy$ with $deg(v)=2$ and $deg(y)=3$, and hence we are in Case 1.  

So let us assume that $deg(y) \geq 4$.  Now since $deg(x)=2$, $x$ has another neighbour besides $w$ --- let this vertex be $z$.  Then $w$ is adjacent to $z$, because the neighbours of a vertex of degree 2 in a MOP must be adjacent, and $wz$ is a chord..  Now since $w$ is adjacent to $z$ and to $y$, $zy$ must be an edge.  Now since $n \geq 6$, $deg(z) \geq 4$.  Recall that $deg(w)=4$ and $deg(y) \geq 4$.  Then if $deg(y) \leq deg(z)$, then $vwyz$ is a dmp of length 4.  Otherwise $deg(y) > deg(z)$, then $vwzy$ is a dmp of length 4.  

\begin{figure}[h!]
\centering
\includegraphics{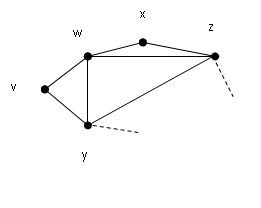}
\caption{Case 2}
\end{figure}

Hence, in both cases there is always a dmp of length at least 4.

For the upper bound, clearly for $n \geq 2$, $mp(G) \leq n-1$ since there are always at least two nonadjacent vertices of degree 2, and these cannot both be included in a degree monotone path.  The upperbound is realised by the MOP obtained by starting with a cycle graph $C_n$ on $n$ vertices, and choosing any vertex $v$ on the cycle and connecting it to every other vertex.  We call this MOP $F_n$, as shown in Figure \ref{fan}.  This is clearly a MOP, with the two neighbours of $v$ on the cycle having degree 2, $v$ having degree $n-1$ and all other vertices are of degree 3.  Starting from one of the vertices of degree 2, one can traverse all vertices of degree 3, and then move to $v$, giving a degree monotone path of length $n-1$.

\begin{figure}[H] \label{fan}
\centering
\includegraphics{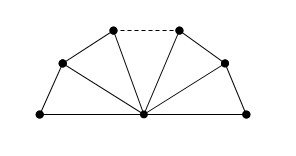}
\caption{$F_7$} \label{fan}
\end{figure}

\medskip
\noindent 2. \indent
The following construction gives a MOP with $mp(G)=4$ for $|V(G)| \geq 6$ and  $|V(G)| \pmod 4 \not=3$.  We start with the graph $F_r$, $ r \geq 5$, with vertex $v$ having degree $r-1$, and all the other vertices are labelled $v_i$, $1 \leq i \leq r-1$, with $v_1$ and $v_{r-1}$ being the neighbours of $v$ with degree 2.    To this graph we add another $\lfloor \frac{r}{3} \rfloor$  vertices of degree 2 as follows:
\begin{enumerate}
\item{ For $r=0,1\pmod 3$,  the first of these added vertices is connected to $v_1$ and $v_2$, the second  to $v_2$ and $v_3$, and so on, so that in general, the $i^{th}$ such vertex is connected to $v_{3i-2}$ and $v_{3i-1}$, for $1 \leq i \leq \lfloor \frac{r}{3} \rfloor$.}
\item {For $r=2 \pmod 3$,  the first of these added vertices is connected to $v_2$ and $v_3$, the second (if $r > 5$) to $v_5$ and $v_6$, and so on, so that in general, the $i^{th}$ such vertex is connected to $v_{3i-1}$ and $v_{3i}$, for $1 \leq i \leq \lfloor \frac{r}{3} \rfloor$.}
\end{enumerate}

Thus each such graph has $r+\lfloor \frac{r}{3} \rfloor$ vertices.  Figure \ref{Cwings}  shows such a graph  with $r=7$ and hence $|V(G)|=9$.  The possible degree sequences of degree monotone paths are $2,3,4,x$;  $2,5,5,x$ and $3,5,5,x$, where $x$ is the degree of the central vertex as shown in Figure \ref{Cwings}.

\begin{figure}[H]
\centering
\includegraphics{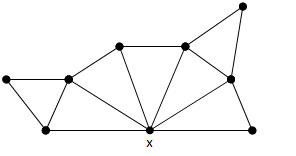}
\caption{$mp(G)=4$} \label{Cwings}
\end{figure}

\noindent 3. \indent

We now construct a family of MOPs with $mp(G)=4$ and with no regular edges.  The construction is quite similar to that given in part 2.  Again we start with $F_r$  , for $r \geq 7$ and label it as in part 2.  We now add  vertices as follows:
\begin{enumerate}
\item{ If $r= 0 \pmod 4$, we add $ \frac{r}{2} $ vertices --- these are added in pairs:  the first pair of vertices is added by connecting the first vertex to $v_1$ and $v_2$, and the second to $v_2$ and $v_3$. In general, a pair of vertices are added by connecting the first vertex to $v_{4i-3}$ and $v_{4i-2}$ and the second to $v_{4i-2}$ and $v_{4i-1}$  for $1 \leq i \leq \frac{r}{4}$.}
\item{ If $r= 1 \pmod 4$, we add $\lfloor \frac{r}{2} \rfloor$ vertices --- these are added in pairs, as for case 1.}
\item{ If $r= 2 \pmod 4$, we add $ 2\lfloor \frac{r}{4} \rfloor $ vertices --- these are added in pairs:  the first pair of vertices is added by connecting the first vertex to $v_2$ and $v_3$, and the second to $v_3$ and $v_4$. In general, a pair of vertices are added by connecting the first vertex to $v_{4i-2}$ and $v_{4i-1}$ and the second to $v_{4i-1}$ and $v_{4i}$  for $1 \leq i \leq \lfloor \frac{r}{4} \rfloor$.}
\item{ If $r= 3 \pmod 4$, we add $\lfloor \frac{r}{2} \rfloor$ vertices, which is an odd number.  The first $\lfloor \frac{r}{2} \rfloor-1 $ vertices are added in pairs:  the first pair of vertices is added by connecting the first vertex to $v_1$ and $v_2$, and the second to $v_2$ and $v_3$. In general, a pair of vertices are added by connecting the first vertex to $v_{4i-3}$ and $v_{4i-2}$ and the second to $v_{4i-2}$ and $v_{4i-1}$  for $1 \leq i \leq \lfloor \frac{r}{4} \rfloor$.  The last vertex is connected to the vertices $v_{r-2}$ and $v_{r-1}$.}
\end{enumerate}

Thus, for $r=0,1,2 \pmod 4$, each such graph has $r+\lfloor \frac{r}{2} \rfloor$ vertices, while for $r=3 \pmod 4$, the number of vertices is $r+2\lfloor \frac{r}{4} \rfloor$ .   One can see that $mp(G)=4$ and no edge is regular.   Figure \ref{YCwheels}  shows such a graph  with $r=9$ and hence $|V(G)|=13$.  The possible degree monotone sequences are $2,3,5,x$; $2,4,5,x$ and  $3,4,5,x$.  So the reader can see also the effect of non-regular edges, as these sequences are strictly monotone increasing, justifying the consideration of the case with no regular edge which forces any degree monotone path to be strictly monotone.
 
\begin{figure}[H]
\centering
\includegraphics{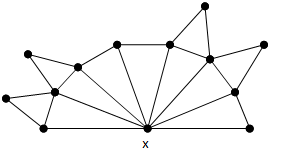}
\caption{$mp(G)=4$ with no regular edges} \label{YCwheels}
\end{figure}

\end{proof}

\subsection{Maximal Planar Graphs}

 For maximal planar graphs the situation is different as we have two constructions for arbitrarily large maximal planar graphs $G$ having $\chi(G)=mp(G)=4$.  
\medskip

\emph{Construction 1}
\medskip

The graph in Figure \ref{MP1} is maximal planar and $mp(G)=4$.  We start with $F_r$ for $r \geq 7$, and $r= 1 \pmod 3$, and label it as in Theorem \ref{dm4}.  We add another vertex and connect it to all the vertices of degree 2 and 3.  Let us call this vertex $y$.  We then add vertices of degree 3 as follows: the first vertex is connected to $y$, $v_1$ and $v_2$,  the next is connected to $v$,$v_4$ and $v_5$, and so on up to $v_{r-2}$.  So in general such a vertex is connected to $y$, $v_{3i-2}$ and $v_{3i-1}$, for $1 \leq i \leq   \frac{r-1}{3}$.  Finally, we connect $v_1$ to $v_{r-1}$.  The graph has $r+1 +\frac{r-1}{3} = \frac{4r+2}{3}$ vertices, and it is clearly maximal planar.  The longest possible degree monotone pathshave degree sequences $3,5,5,x$; $3,5,5,y$; $4,5,5,x $ and $4,5,5,y$,  where $x$ is the degree $r-1$, as labelled in Figure \ref{MP1}.  The situation is similar if one starts from a vertex of degree 4.  Hence $mp(G)=4$ and clearly $\chi(G)=4$, since the graph contains $K_4$.

\begin{figure}[H]
\centering
\includegraphics{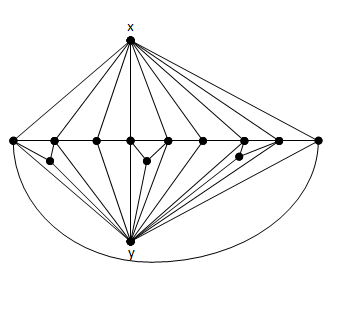}
\caption{Maximal planar graph with $\chi(G)=mp(G)=4$} \label{MP1}
\end{figure}

\emph{Construction 2}
\medskip

We start with a graph, which we label $A_i$ --- we take $C_4$ and connect two non-adjacent vertices, so that there are two vertices of degree 2, labelled $x_i$ and $y_i$, and two vertices of degree 3, labelled $z_i$ and $w_i$.  Then we take $k \geq 3$ copies of $A_i$ and join them together by merging $y_i$ with $x_{i+1}$ into a single vertex of degree 4.  Now we add two new vertices $v_1$,$v_2$ and $v_3$ --- we connect $v_1$ to $x_1$, $v_2$ and $v_3$; we connect $v_2$ to all $x_i$ and $z_i$, and to $y_k$;  finally we connect $v_3$ to all $x_i$ and $w_i$, and to $y_k$.

It is clear that this graph is maximal planar.  If we start a path from a vertex of degree 4, there is another vertex of degree 4 adjacent to it.  The possible degree sequences of degree monotone paths are $4,4,6,v_2$; $4,4,6,v_3$; $4,4,5,v_2$; $4,4,5,v_3$.  Thus $mp(G)=4$.  We now show that a proper colouring of the graph requires at least four colours.  So suppose we try to colour using colours 1,2 and 3.  Let the vertices $z_i$ take colour 1, and the vertices $w_i$ take colour 2.  The vertices $x_i$ and $y_i$ must take colour 3.  But then $v_2$ must take colour 2 and $v_3$ must take colour 1, and therefore $v_1$ has neighbours of all three colours, hence it must take a new colour 4.  Hence $\chi(G)=4$.

\begin{figure}[H]
\centering
\includegraphics{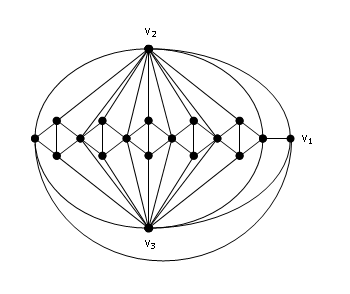}
\caption{Maximal planar graph with $\chi(G)=mp(G)=4$} \label{MP2}
\end{figure}

\section{Extremal Graphs for $mp(G)$ and Turan numbers}

We now turn our attention to $f(n,k)$.  Recall that we define $ f(n,k) =\max\{ |E(G)|: |V(G)| = n, mp(G) < k\}$, that is the maximum number of edges in a graph on $n$ vertices with no degree monotone path of length $k$, and that  the Turan number $t(n,k)  = \max \{|E(G)|: |V(G)| = n, G \mbox{ contains no copy of } K_k\}$.  Trivially $t(n,2)=f(n,2)=0$, hence we assume $k \geq 3$.

\begin{propskip} \label{f_leq_t}
For $k \geq 3$, $f(n,k) \leq t(n,k)$.
\end{propskip}

\begin{proof}
Suppose $G$ has $t(n,k)+1$ edges,  then by the definition of $t(n,k)$, $G$ contains $K_k$. Now we know that $mp(G) \geq \chi(G) \geq \omega(G) \geq k$, so there is a degree monotone path of length at least $k$.  Therefore $f(n,k) \leq t(n,k)$.
\end{proof}

\begin{propskip} \label{f_leq_t-1}
For $k \geq 4$ and $n \geq k+1$, $f(n,k) \leq t(n,k)-1$.
\end{propskip}

\begin{proof}
For $k \geq 4$, the unique Turan graph is a complete $(k-1)$-partite graph with $k-1 \geq 3$ parts all of order $\lfloor \frac{n}{k-1} \rfloor$ or $\lceil \frac{n}{k-1} \rceil$.

Now since $k \geq 4$ and $n \geq k+1$, we must have one of the following two scenarios:

\noindent Case 1:  There are at least two parts of size $\lceil \frac{n}{k-1} \rceil \geq 2$.  We observe that the vertices in these parts have degree less than or equal to those in the parts of size $\lfloor \frac{n}{k-1} \rfloor$.  Let $V_1$ and $V_2$ be two classes of size $\lceil \frac{n}{k-1} \rceil $ -- Thus vertices in these two parts have the same degree, and we can take a degree monotone path of at least four vertices from these parts, and then we can continue taking consecutive vertices from all the other classes, by first taking vertices from the larger class, and then moving onto the smaller classes to get a degree monotne path with at least $k-3+4=k+1$ vertices, giving $mp(G) \geq k+1$.

\medskip

\noindent Case 2:  There is exactly one part of size $\lceil \frac{n}{k-1} \rceil$ and $k-2$ parts of size $\lfloor \frac{n}{k-1} \rfloor \geq 2$.  We can start with a vertex in the large part, and then move to the smaller parts, in which all vertices have the same degree and thus we can include all vertices to give a degree monotone path.  Then \[mp(G) \geq (k-2) \lfloor \frac{n}{k-1} \rfloor + 1 \geq  2(k-2)+1 = 2k-3 \geq k+1 \mbox{ since } k \geq 4.\]  Thus in both cases, $mp(G) \geq k+1$, and hence $f(n,k) \leq t(n,k)-1$.
\end{proof}

We now characterise $f(n,3)$.

\begin{proposition}
\mbox{} \begin{enumerate}
\item{$f(n,3)=\frac{n^2}{4}-1=t(n,3)-1$ for $n=0 \pmod 2$.}
\item{$f(n,3)=\frac{n^2-1}{4}=t(n,3)$ for $n= 1 \pmod 2$.}
\end{enumerate}
\end{proposition}

\begin{proof}

For $n=0 \pmod 2$, the Turan graph which achieves $t(n,3)$ is $K_{\frac{n}{2},\frac{n}{2}}$, which has $\frac{n^2}{4}$ edges and $mp(G)=n$, since the graph is regular and Hamiltonian.   So let us consider the graph $K_{\frac{n}{2}+1,\frac{n}{2}-1}$ - this graph has $\frac{n^2}{4}-1$ edges and $mp(G)=2$.  Any graph with more edges either contains $K_3$, and hence a degree monotone path with at least three vertices, or is the Turan graph $K_{\frac{n}{2},\frac{n}{2}}$ for which $mp(G)=n$, hence the lower bound is sharp.

For $n = 1 \pmod 2$, the Turan graph which achieves $t(n,3)$ is $K_{\frac{n+1}{2},\frac{n-1}{2}}$, which has $\frac{n^2-1}{4}$ edges.  It is easy to see that $mp(K_{\frac{n+1}{2},\frac{n-1}{2}})=2$, hence even in this case the lower bound is sharp by Proposition \ref{f_leq_t}.
\end{proof}

\begin{proposition}\label{A_n}
Let $A_{n,k}$ be the family of sequences of integers  $1 \leq a_1 < a_2 < \ldots < a_{k-1}$ such that \[\sum_{i=1}^{i=k-1}a_i=n .\] Let \[g(n,k)=\max\{ \sum_{1 \leq i < j \leq k-1} a_i a_j:    \mbox{ over all sequences in  } A_{n,k}\}.\] Then $f(n,k) \geq g(n,k).$

\end{proposition}

\begin{proof}
From each  sequence in $A_{n,k}$ we construct a complete $(k-1)$-partite graph $K_{a_1,a_2,\ldots,a_{k-1}}$. The number of edges of this graph is precisely \[\sum_{1 \leq i < j \leq k-1} a_i a_j.\]  The degrees of the vertices in class $a_j$ are precisely $n-a_j$, so distinct classes have different vertex degrees.  A degree monotone path can clearly have exactly one vertex from each class, hence $mp(G) = k-1$.

Thus the class of such graphs gives the lower bound for $f(n,k)$, proving the proposition.
\end{proof}

Using Propositions \ref{f_leq_t} and \ref{A_n},  we can now show that the values of $f(n,k)$ and $t(n,k)$ are in fact quite close.  It is known that $t(n,k) \leq \frac{n^2(k-2)}{2(k-1)}$  (see \cite{bollobas2004extremal}).

\begin{theorem} \label{c_1_c_2}
For $k \geq 3$ and $n \geq \frac{(k-1)(k+2)}{2}$, $ t(n,k)-c(k) \leq g(n,k) \leq f(n,k) \leq t(n,k)$, where \[c(k) \leq \frac{k^3+5k+3}{24}.\]
\end{theorem}

\begin{proof}
We assume $n \geq \frac{(k-1)(k+2)}{2}$ ---  let us choose $r$ so that  \[\sum_{i=0}^{k-2} r+i < n \leq \sum_{i=1}^{k-1} r+i =\frac{(2r+k)(k-1)}{2}.\]  Let $t=\frac{(2r+k)(k-1)}{2}-n$ and note that \[r=\frac{\frac{(2(n+t)}{k-1}-k}{2}= \frac{n+t}{k-1}- \frac{k}{2} \geq 1\] by the choice of $n$.  Now let us subtract 1 from the $t$ smallest values of the sequence on the right hand side of the above equation, and call the resulting terms $a_i$. Then $a_i=r+i-1$ for $ 1 \leq i \leq t$ and $a_i=r+i$ for $t+1 \leq i \leq k-1$, and  \[\sum_{i=1}^{k-1} a_i=n.\]   We now consider the $(k-1)$-partite graph with parts of sizes $a_1,a_2,\ldots,a_{k-1}$.  The number of edges of the graph is given by \[|E(G)|=\frac{\sum_{v \in G}deg(v)}{2}= \frac{\sum_{i=1}^{k-1}a_i(n-a_i)}{2}\]\[=n\frac{\sum_{i=1}^{k-1}a_i}{2} - \frac{\sum_{i=1}^{k-1}a_i^2}{2}= \frac{n^2}{2}-\frac{\sum_{i=1}^{k-1}a_i^2}{2}.\]  Now let us consider different values of $t$.

\noindent Case 1. \indent $t=0$.

\[\frac{n^2}{2}-\frac{\sum_{i=1}^{k-1}a_i^2}{2}=\frac{n^2}{2}-\frac{\sum_{i=1}^{k-1}(r+i)^2}{2}=\frac{n^2}{2}-\frac{\sum_{i=1}^{k-1}(r^2 + 2ir + i^2)}{2}\]\[=\frac{n^2}{2} - \frac{(k-1)r^2}{2}-\frac{rk(k-1)}{2}-\frac{k(k-1)(2k-1)}{12}\]\[=\frac{n^2}{2} - \frac{k-1}{2} \left [r^2+rk+\frac{k(2k-1)}{6} \right].\]  Now, when $t=0$, $r=\frac{n}{k-1}-\frac{k}{2}$, hence we get 

\[\frac{n^2}{2} - \frac{k-1}{2} \left [r^2+rk+\frac{k(2k-1)}{6} \right]=\frac{n^2}{2} - \frac{k-1}{2} \left [r(r+k)+\frac{k(2k-1)}{6} \right]\]
\[=\frac{n^2}{2} - \frac{k-1}{2} \left [\left( \frac{n}{k-1}-\frac{k}{2} \right) \left (\frac{n}{k-1}-\frac{k}{2}+k \right )+\frac{k(2k-1)}{6} \right]\]
\[=\frac{n^2(k-2)}{2(k-1)} -\frac{k(k-1)(k-2)}{24} \]


\bigskip
\noindent Case 2. \indent $1 \leq t \leq k-2$.

\[\frac{n^2}{2}-\frac{\sum_{i=1}^{k-1}a_i^2}{2}=\frac{n^2}{2}-\frac{\sum_{i=1}^{t}(r+i-1)^2}{2}-\frac{\sum_{i=t+1}^{k-1}(r+i)^2}{2}\]\[=\frac{n^2}{2}-\frac{\sum_{i=1}^{t}(r^2+2r(i-1) +(i-1)^2)}{2} - \frac{\sum_{i=t+1}^{k-1}(r^2+2ri+i^2)}{2}\]\[=\frac{n^2}{2}-\frac{tr^2}{2} - \frac{rt(t-1)}{2}-\frac{t(t-1)(2t-1)}{12}-\frac{(k-t-1)r^2}{2}-r\left(\frac{k(k-1)}{2}- \frac{t(t+1)}{2}\right)- \frac{\sum_{i=t+1}^{k-1}i^2}{2}\]\[=\frac{n^2}{2}-\frac{r^2(k-1)}{2} +rt - \frac{rk(k-1)}{2} -\frac{t(t-1)(2t-1)}{12}- \frac{k(k-1)(2k-1)}{12} + \frac{t(t+1)(2t+1)}{12}\]\[=\frac{n^2}{2} - \frac{r(k-1)(r+k)}{2} +rt - \frac{k(k-1)(2k-1)}{12} +\frac{t}{12}[2t^2+3t+1-(2t^2-3t+1)]\]\[=\frac{n^2}{2} - \frac{r(k-1)(r+k)}{2} +rt - \frac{k(k-1)(2k-1)}{12} +\frac{t^2}{2}\]

 Now replacing $r$ by $\frac{n+t}{k-1}-\frac{k}{2}$, we get
\[=\frac{n^2}{2} - \frac{(k-1)}{2}\left[\left(\frac{n+t}{k-1}-\frac{k}{2}\right)\left(\frac{n+t}{k-1}-\frac{k}{2}+k\right)\right] +t\left(\frac{n+t}{k-1}-\frac{k}{2} \right)- \frac{k(k-1)(2k-1)}{12} +\frac{t^2}{2}\] 
\[=\frac{n^2(k-2)}{2(k-1)}-\frac{k(k-1)(k-2)}{24}+\frac{kt^2-k^2t+1}{2(k-1)}\]

Now this expression is minimum when $t=\frac{k}{2}$, so this gives \[=\frac{n^2(k-2)}{2(k-1)}-\frac{k(k-1)(k-2)}{24}+\frac{k(\frac{k}{2})^2-k^2(\frac{k}{2})+1}{2(k-1)}\]
\[=\frac{n^2(k-2)}{2(k-1)} -\frac{k^3+5k+3}{24}+\frac{9}{24(k-1)}\]\[\geq \frac{n^2(k-2)}{2(k-1)} -\frac{k^3+5k+3}{24} \mbox{, since } k \geq 2\]

Therefore \[|E(G)| \geq t(n,k) -\frac{k^3+5k+3}{24} \] for any value of $t$  and the theorem is proved.
\end{proof}

\section{Nordhaus-Gaddum results for $mp(G)$}

We now turn our attention to the value of $mp(G)$ and $mp(\overline{G})$, and present a Nordhaus-Gaddum type of result.

\begin{proposition}

Let $G$ be a graph on $n$ vertices.
\begin{enumerate}
\item{For every such graph, $2\sqrt{n} \leq mp(G) + mp(\overline{G}) \leq 2n$.}
\item{There exist such graphs for which $mp(G) + mp(\overline{G}) =2n$.}
\item{There exist such graphs for which $mp(G) + mp(\overline{G}) =\frac{5\sqrt{n}}{2}$.}
\end{enumerate}
\end{proposition}

\begin{proof}
\mbox{}\\

\noindent 1. \indent Clearly, $mp(G) \leq n$, and hence $mp(G) + mp(\overline{G}) \leq 2n$.  Also, by Proposition \ref{mp_chi} and the Nordhaus-Gaddum theorem \cite{nordhaus1956complementary}, $mp(G) + mp(\overline{G}) \geq \chi(G) + \chi(\overline{G}) \geq 2\sqrt{n}$.
\bigskip

\noindent 2. \indent We show that the upperbound is sharp --- for $n \geq 5$, we observe the $K_n$ has a Hamiltonian cycle $C$, such that when we delete $C$, we are left with a regular graph $G$ of degree $(n-1)-2=n-3 \geq \frac{n}{2}$.  Hence $G$ is Hamiltonian by Dirac's Theorem (in \cite{dirac1952some}).  It follows that $mp(C)+mp(G) = mp(\overline{G}) +  mp(G) = n + n = 2n$, as required.
\bigskip

\noindent 3. \indent  Let $n$ be an even square number.  Let $G$ be the graph on $n$ vertices made up of vertex disjoint cliques of size $\frac{\sqrt{n}}{2}$ up to $\frac{3\sqrt{n}}{2}$, except for the order $\sqrt{n}$.  Clearly $mp(G)$ is equal to the size of the largest clique, that is $mp(G)=\frac{3\sqrt{n}}{2}$.

The complement graph $\overline{G}$, is a complete $\sqrt{n}$-partite graph with vertices in distinct parts having different degrees, and hence $mp(\overline{G})=\sqrt{n}$.

Therefore, $ mp(G)+mp(\overline{G})= \sqrt{n}+ \frac{3\sqrt{n}}{2} = \frac{5\sqrt{n}}{2}$.

If $n$ is an odd square number, we take a similar construction with disjoint cliques of size $\left \lceil \frac{\sqrt{n}}{2} \right \rceil$ up to $\left \lfloor \frac{3\sqrt{n}}{2} \right \rfloor$ (including $\sqrt{n}$).  So again $mp(G)=\left \lfloor \frac{3\sqrt{n}}{2} \right \rfloor$.

The complement is again a complete $\sqrt{n}$-partite graph with vertices in distinct parts having different degrees, and hence $mp(\overline{G})=\sqrt{n}$.

Therefore, $ mp(G)+mp(\overline{G})= \sqrt{n}+ \left \lfloor \frac{3\sqrt{n}}{2} \right \rfloor= \left \lfloor\frac{5\sqrt{n}}{2} \right \rfloor$.

\end{proof}

\section{Conclusion}

The results which we have presented lead to a few open problems.

We have been able to show that, for all maximal outerplanar graphs $G$, $mp(G) \geq 4 > 3 =\chi(G)$, therefore improving on the lowerbound obtained using the Gallai-Roy Theorem.  Our knowledge for maximal planar graphs $G$ is, however, not so complete.  Our examples show that the lower bound can be attained, but they all have chromatic number equal to 4.  The existence of arbitrarily large maximal planar graphs $G$ with $mp(G)=3=\chi(G)$ is still open.

When we were investigating the relationship between $f(n,k)$ and the Turan numbers we defined the parameter $g(n,k)$ defined over all sequences of distinct $k-1$ integers summing to $n$, and we showed that $f(n,k) \geq g(n,k)$.  Although we have not been able to show that equality holds, we do conjecture that in fact, $f(n,k)=g(n,k)$ for $n \geq \frac{(k-1)(k+2)}{2}$.

We have also obtained bounds on $mp(G)$ of Norhaus-Gaddum type, but while we were able to show that the upper bound is attained, we do not know whether the lower bound of $2\sqrt{n}$ for $mp(G) + mp(\overline{G})$ is sharp.

Finally, we mentioned in the introduction, a fourth question dealing with $mp(G)$:  how does the value of $mp(G)$ change when $G$ is subjected to various graph operations such as edge/vertex deletion or addition?  A forthcoming paper \cite{CLZ5dmp} will deal with this question.
\bibliographystyle{plain}
\bibliography{mopbib}
\end{document}